\documentclass[12pt]{amsart}
\usepackage{amsmath,amsthm,amsfonts,amssymb}

\newcommand{\bbC}{\mathbb{C}}
\newcommand{\bbZ}{\mathbb{Z}}

\newcommand{\bbQ}{\mathbb{Q}}

\newcommand{\bbF}{\mathbb{F}}

\newcommand{\bbN}{\mathbb{N}}
\newcommand{\cA}{\mathcal{A}}

\newcommand{\Span}{\operatorname{Span}}

\newcommand{\m}{\mathfrak{m}}

\newcommand{\Aut}{\operatorname{Aut}}

\newcommand{\Spec}{\operatorname{Spec}}
\newcommand{\Gal}{\operatorname{Gal}}
\newcommand{\sep}{\operatorname{sep}}
\newcommand{\tor}{\operatorname{tor}}
\newcommand{\ch}{\operatorname{char}}

\newcommand{\GL}{\mathrm{GL}}

\newtheorem{thm}{Theorem}

\newtheorem{cor}[thm]{Corollary}

\newtheorem{prop}[thm]{Proposition}

\newtheorem{conj}[thm]{Conjecture}

\newtheorem*{claim*}{Claim}
\begin{document}

\title[Hales-Jewett and Field Arithmetic]{Some Applications of the Hales-Jewett Theorem to Field Arithmetic}
\author{Bo-Hae Im and Michael Larsen}
\date{\today}
\address{Department of Mathematics, Chung-Ang University, 221, Heukseok-dong, Dongjak-gu, Seoul, 156-756, South Korea}\email{bohaeim@gmail.com}
  \address{Department of Mathematics, Indiana University, Bloomington,
Indiana 47405, USA} \email{mjlarsen@indiana.edu}
\thanks{Bo-Hae Im was supported by the National Research Foundation of Korea Grant funded by the Korean Government(MEST) (NRF-2011-0015557). Michael Larsen was partially supported by NSF grants DMS-0800705 and DMS-1101424.}

\begin{abstract}
Let $K$ be a field whose absolute Galois group is finitely generated.  If $K$ neither finite nor of characteristic $2$, then every hyperelliptic
curve over $K$ with all of its Weierstrass points defined over $K$
has infinitely many $K$-points.  If, in addition, $K$ is not locally finite, then every elliptic curve over $K$
with all of its $2$-torsion rational has infinite rank over $K$.  These and similar results are deduced from the Hales-Jewett theorem.
\end{abstract}

\maketitle

\section{Introduction}

This paper is motivated by the following conjecture:

\begin{conj}
\label{Conjecture}
If $K$ is a field which is not locally finite and such that $\Gal(K^{\sep}/K)$ is topologically finitely generated,
then every non-trivial elliptic curve $E/K$ has
infinite rank over $K$.
\end{conj}

Frey and Jarden \cite{FJ} used probabilistic methods to prove that for a given $E$,
the rank is infinite for ``most'' such fields $K$ over which $E$ is defined.  In a series of papers \cite{Larsen-LMS,Im-PAMS,Im-CJM,IL-AJM},
the authors used methods of Diophantine geometry to prove this result whenever $\Gal(K^{\sep}/K)$ is topologically cyclic (assuming the characteristic of $K$ is not $2$.)
A different line of investigation \cite{Im-TAMS,IB-CJM} gave similar results using Heegner point methods.
Neither approach offers much hope of progress in the case where $\Gal(K^{\sep}/K)$ requires two or more generators.
A third method was suggested by the Tim and Vladimir Dokchitser \cite{DD}, who proved the result for any characteristic zero field $K$, provided
the $j$-invariant of $E$ is an algebraic number.  Their proof is conditional, however, requiring some part of the Birch-Swinnerton-Dyer conjecture,
either the equality of analytic and
Mordell-Weil ranks or the finiteness conjecture for Tate-Shafarevich groups.  In this paper, we introduce a fourth method, based on the Hales-Jewett theorem in combinatorics,
which offers unconditional results for any number of generators, though at present an additional hypothesis on $E$ is needed.

We recall that a field $K$ is said to be \emph{ample} if every smooth curve $X/K$ with a non-singular $K$-point satisfies $|X(K)| = \infty$.
A conjecture of Junker and Koenigsmann \cite{JK} asserts that every infinite field with finitely generated Galois group is ample.
This would imply Conjecture \ref{Conjecture} when $K$ is of characteristic zero \cite[Proposition 2.4]{BF}.  In this paper, we
prove that certain kinds of pointed curves (namely split hyperelliptic curves) have infinitely many points over any infinite field
with finitely generated Galois group.  This supports the Junker-Koenigsmann conjecture, though of course split hyperelliptic curves are rather special.
In general, we can prove infinite rank for those abelian varieties which contain split hyperelliptic curves, again, a rather special
set of examples.

\section{Density for hyperelliptic curves}

Let $K$ be a field not of characteristic $2$.
We fix a separable closure, which we denote $K^{\sep}$ and
let $G_K := \Gal(K^{\sep}/K)$.

By a \emph{split hyperelliptic curve}, we mean a curve $X/K$
which has an affine open set of the form
$$\Spec K[x,y]/(y^2-(x-a_1)\cdots(x-a_{2g+2})),$$
where $g\in\bbN$ and the $a_i$ are pairwise distinct elements of $K$.

\begin{prop}
\label{MainProp}
If $K$ is not of characteristic $2$ and
$G_K$ is  finitely generated (as a topological group), then for every split hyperelliptic
curve $X/K$, there exist a finite subset $C\subset K$ and a finite set $L\subset K[t]$ of
degree $1$ polynomials such that for all $c\in K\setminus C$ there exist $\lambda\in L$
and a point in $X(K)$ whose $x$-coordinate is $\lambda(c)$.
\end{prop}

\begin{proof}
As $G_K$ is finitely generated, by Kummer theory, $Q := K^\times/(K^\times)^2$ is finite.
For each non-zero $c\in K$, we write $[c]$ for the class of $c$ in $Q$.
Let $[1,n]$ denote the set of positive integers from $1$ to $n$.
By the Hales-Jewett theorem, there exists $N$ such that for every partition of
$Z := [1,2g+2]^N$ into sets $Y_q$ indexed by elements $q\in Q$, there exist $q\in Q$ and
functions $f_1,\ldots,f_N\colon [1,2g+2]\to [1,2g+2]$ such that the following conditions hold.
\begin{itemize}
\item Each $f_i$ is either constant or the identity function.
\item At least one $f_i$ is non-constant.
\item For all $j\in[1,2g+2]$, $(f_1(j),f_2(j),\ldots,f_N(j))\in Y_q$.
\end{itemize}

Let $X$ have an affine open set of the form
$$y^2 = (x-a_1)\cdots(x-a_{2g+2}),$$
where $g\in\bbN$ and the $a_i$ are pairwise distinct elements of $K$.
We fix elements $b_1,\ldots,b_N\in K$ such that no non-trivial subset of the $b_i$ sum to zero.
This is possible since $K$ is infinite.
Define
$$C := \{b_1a_{i_1}+\cdots+b_N a_{i_N}\mid i_1,\ldots,i_N\in [1,2g+2]\}.$$
For all $c\in K\setminus C$, we can define a partition $\{Y_{c,q}\}_{q\in Q}$ of $Z$ by
$$Y_{c,q} := \{(i_1,\ldots,i_N)\in Z\mid [c-b_1a_{i_1}-\cdots-b_N a_{i_N}] = q\}.$$
Let $f_{c,1},\ldots,f_{c,N}$ denote functions satisfying the  Hales-Jewett
theorem for the partition $\{Y_{c,q}\}_{q\in Q}$, and
let $S_c\subset [1,N]$ denote the (non-empty) set of indices $i$
such that $f_{c,i}$ is the identity function.  Explicitly, for $i\in S_c$, $f_{c,i}(j) = j$, while for $i\in [1,N]\setminus S_c$, $f_{c,i}(j) = f_{c,i}(1)$.

We define
$$r_c := \sum_{i\in [1,N]\setminus S_c} b_i a_{f_{c,i}(1)}$$
and
$$s_c := \sum_{i\in S_c} b_i\neq 0.$$
Thus,
\begin{align*}
c-b_1a_{f_{c,1}(j)}-\cdots-b_N a_{f_{c,N}(j)} &= c-\sum_{i\in [1,N]\setminus S_c} b_i a_{f_{c,i}(1)}-  \sum_{i\in S_c} b_i a_j \\
							&= c - r_c - s_c a_j.
\end{align*}

For all $j\in [1,2g+2]$,
$$(f_{c,1}(j),\ldots,f_{c,N}(j))\in Y_{c,q},$$
so
$$q = [c-b_1a_{f_{c,1}(j)}-\cdots-b_N a_{f_{c,N}(j)}] = [c - r_c - s_c a_j].$$
It follows that
$$\prod_{j=1}^{2g+2} \Bigl(\frac{c-r_c}{s_c}-a_j\Bigr) = s_c^{-2(g+1)} \prod_{j=1}^{2g+2} (c - r_c-s_ca_j) \in (K^\times)^2,$$
and therefore that $X(K)$ contains a point $(x,y)$ with $x=(c-r_c)/s_c$.
The proposition now holds for
$$L := \biggl\{\frac{t-\sum_{i\in [1,N]\setminus S} b_i a_{f_{c,i}(1)}}{\sum_{i\in S} b_i}\biggm| \emptyset\neq S\subset [1,N]\biggr\}.$$

\end{proof}

\begin{thm}
If $K$ is infinite and not of characteristic $2$ and $G_K$ is  finitely generated, then for every split hyperelliptic
curve $X/K$, we have $|X(K)| = \infty$.
\end{thm}

\begin{proof}
Let $S$ denote the set of $x\in K$ such that $(x-a_1)\cdots(x-a_{2g+2})\in (K^\times)^2$.
We apply Proposition~\ref{MainProp} to obtain finite sets of constants $C\subset K$ and affine linear functions $L\subset K[t]$
such that for all $c\in K\setminus C$, there exists $\lambda\in L$ such that $\lambda(c)\in S$.
For any $s\in S$ and $\lambda\in L$. there exists at most one $c\in K\setminus C$ such that $\lambda(c)=s$.
Therefore, if $K$ is infinite, $S$ must be infinite as well.
\end{proof}

The following corollary follows immediately.

\begin{cor}
Let $X$ be a hyperelliptic curve over an infinite field $K$ not of characteristic $2$.  Then for all $n\in \bbN$, the subset
$$\{(\sigma_1,\ldots,\sigma_n)\in G_K^n\colon |X(K^{\sep})^{\langle \sigma_1,\ldots,\sigma_n\rangle}| = \infty\}\subset G_K^n$$
has non-empty interior.
\end{cor}

This suggests the following conjecture:

\begin{conj}
If $X$ is any curve over an infinite field $K$, then
$$\{(\sigma_1,\ldots,\sigma_n)\in G_K^n\colon |X(K^{\sep})^{\langle \sigma_1,\ldots,\sigma_n\rangle}| = \infty\}\subset G_K^n$$
has non-empty interior.
\end{conj}

\section{Infinite rank for hyperelliptic Jacobians}

We begin with an extension of a lemma of Silverman \cite{Silverman}.

\begin{prop}
\label{Silverman}
Let $A$ be an abelian variety over a field $K$ which is finitely generated over its prime field.
Then for all $d\in \bbN$
$$\bigcup_{\{L\mid K\subset L\subset K^{\sep},\; [L:K] \le d\}} A(L)_{\tor}$$
is finite.
\end{prop}

\begin{proof}
Let $K_0 = K$ and $A_0 = A$.  We claim there exists a sequence $K_1,\ldots,K_m$ of fields, a sequence $A_1/K_1,\ldots,A_m/K_m$ of abelian varieties,
a sequence $(R_0,(\pi_0)),\ldots,(R_{m-1},(\pi_{m-1}))$ of valuation rings, and a sequence $\cA_0/R_0,\ldots,\cA_{m-1}/R_{m-1}$ of abelian schemes, with the following
properties.
\begin{itemize}
\item $K_m$ is finite.
\item Each $R_i$ is a localization of a finitely generated $\bbZ$-algebra.
\item The fraction field of $R_i$ is isomorphic to $K_i$.
\item The residue field of $R_i$ is isomorphic to $K_{i+1}$.
\item Via these isomorphisms the generic fiber and special fiber of $\cA_i$ are isomorphic to $A_i$ and $A_{i+1}$
respectively.
\end{itemize}

Indeed, every $K$ is the field of fractions of some finitely generated $\bbZ$-algebra $B\subset K$, and we prove the claim by induction on
the Krull dimension of $B$.  If $B$ is a field which is finitely generated over $\bbZ$, it is a finite field, so the base case is trivial.
As $\bbZ$ is an excellent ring \cite[IV~7.8.3~(iii)]{EGA}, the same is true for $B$ \cite[IV~7.8.3~(ii)]{EGA},
so the regular locus of $\Spec B$ is open \cite[IV~7.8.3~(iv)]{EGA} and non-empty since $B$ is an integral domain.
Replacing $B$ by
$B[1/b]$ for a suitable $b\in B$, we may assume that $B$ is regular.  The abelian variety $A$ extends to
an abelian scheme over some non-empty open subset of $\Spec B$, so inverting an additional non-zero element of $B$,
we may assume $\cA$ is an abelian scheme over $B$.  Let $P$ denote any prime ideal, minimal among non-zero primes of $B$.
Let $R_0 := B_P$, and let $\cA_0$ denote the abelian scheme obtained from $\cA$ by base change from $B$ to $B_P$.
As $B$ is regular, $B_P$ is a regular local domain of Krull dimension $1$, i.e., a discrete valuation ring.
Let $K_1$ be the residue field of $R_0$, or equivalently, the field of fractions of $B/P$, which is a finitely generated
$\bbZ$-algebra of Krull dimension less than that of $B$.
Let $A_1$ denote the special fiber of $\cA_0$.
By the induction hypothesis, we can
now construct the remaining terms of the sequence.

Given any extension $L/K$ of degree $\le d$, we iteratively construct extensions $L_i/K_i$ of degree $\le d$
and discrete valuation rings $S_i\supset R_i$ with fraction field $L_i$ and residue field $L_{i+1}$.  The integral
closure $S'_i$ of $R_i$ in $L_i$ is a semi-local Dedekind domain; localizing at any maximal ideal, we obtain $S_i$.
We define $L_{i+1}$ to be the residue field of $S_i$.  Since $S'_i$ is an integral domain, it is a torsion-free as $R_i$-module.
As $R_i$ is finitely generated over the excellent ring $\bbZ$, $S'_i$ is a finitely generated
$R_i$-module \cite[IV~7.8.3~(vi)]{EGA}.  By the classification of modules over a discrete valuation ring, $S'_i\cong R_i^{[L_i:K_i]}$.
Thus $S'_i/(\pi_i)$ is a vector space of dimension $[L_i:K_i]$ over
$R_i/(\pi_i) = K_{i+1}$.  The quotient $L_{i+1}$ of $S_i$ by its maximal ideal is a quotient $K_{i+1}$-module
of $S'_i/(\pi_i)$, and it follows that $[L_{i+1}:K_{i+1}] \le [L_i:K_i]$.

It is well-known that if $n$ is not divisible by the characteristic of $L_{i+1}$, no
non-trivial $n$-torsion point of $A_i(L_i) = \cA_i(L_i) = \cA_i(S_i)$ lies in the kernel of
the specialization homomorphism to $\cA_i(L_{i+1}) = A_{i+1}(L_{i+1})$.
If $A(L)_{p-\tor}$ and $A(L)_{p'-\tor}$ denote the $p$-power torsion and prime-to-$p$ torsion
respectively, then
$$A(L)_{\tor} = A(L)_{p-\tor}\oplus A(L)_{p'-\tor},$$
and there is an injective homomorphism $A(L)_{p'-\tor}\to A_m(L_m)$.  By the Lefschetz trace formula
$$|A_m(L_m)| \le (1+\sqrt{|L_m|})^{2\dim A} < (1+|K_m|^{d/2})^{2\dim A}.$$
It follows that
$$A(L)_{p'-\tor}\subset A(\bar K)[N],$$
where $N$ is the least common multiple of all integers prime to $p$ and less than or equal to $(1+|K_m|^{d/2})^{2\dim A}$.
To prove the proposition, then, it suffices to show that $|A(L)_{p-\tor}|$ is bounded above.

Let $T_p(A)$ denote the physical $p$-adic Tate module of $A$, i.e.,
$$T_p(A) := \varprojlim_nA(K^{\sep})[p^n].$$
Thus $T_p(A)$ is a free $\bbZ_p$-module with a continuous $G_K$ action.  Let $G_{K,p}$ denote the image
of $G_K$ in $\Aut(T_p(A))\cong \GL_r(\bbZ_p)$.  As $\GL_r(\bbZ_p)$ is virtually $p$-adic analytic, its
closed subgroup $G_{K,p}$ is finitely generated.  Let $G^{\le d}_{K,p}\subset G_{K,p}$
denote the intersection of all kernels of continuous actions of $G_{K,p}$ on sets of cardinality $\le d$.  As
$G_{K,p}$ is finitely generated, there are finitely many such actions,
each corresponding to a homomorphism from $G_{K,p}$ to a symmetric group, and $G^{\le d}_{K,p}$,
as the intersection of the kernels of these homomorphisms,
is of finite index in $G_{K,p}$.

The inverse image of $G^{\le d}_{K,p}$ under the quotient map $G_K\to G_{K,p}$ is an open subgroup of $G_K$ and therefore
of the form $G_M$ for some finite extension $M$ of $K$.  Let $L\subset K^{\sep}$ be any separable extension of $K$ of degree $\le d$.
Suppose $a\in  A(K^{\sep})_{p-\tor}$ is fixed by every element of $G_L$.  Then its orbit under $G_K$ has $\le d$ elements.
Thus, its orbit under $G_{K,p}$ has $\le d$ elements, and it follows that $a$ is fixed by $G^{\le d}_{K,p}$ and therefore
by $G_M$.

Since $M$ is finite over $K$, it is finitely generated over its prime field, so by N\`eron's theorem \cite{Neron},
$A(M)$ is finitely generated.  In particular, for every separable extension $L/K$ of degree $\le d$,
$$A(L)_{p-\tor} = A(K^{\sep})_{p-\tor}^{G_L} \subset A(K^{\sep})_{p-\tor}^{G_M} = A(M)_{p-\tor}$$
is bounded independently of $L$.

\end{proof}

We need the following Chebotarev-type lemma in commutative algebra, and lacking a precise reference, we provide a proof.

\begin{prop}
\label{Chebotarev}
Let $R_0$ be a finitely generated $\bbZ$-algebra and $\phi\colon R_0\to R$ a homomorphism such that $R$ is an infinite, finitely generated $\phi(R_0)$-module.
Then there exist infinitely many maximal ideals $\m$ of $R$ such that
$$|R/\m| = |R_0/\phi^{-1}(\m)| = q < \infty.$$
\end{prop}

\begin{proof}
Replacing $R_0$ by its image in $R$, we may assume without loss of generality that $\phi$ is injective.
By  \cite[V, \S3, Th.~3]{Bourbaki}, $R$ is a Jacobson ring and
if $\m$ is any maximal ideal of $R$, then $R/\m$ is a finite field.
As $R$ is a Jacobson ring, its nilradical $I$ must be an intersection of maximal ideals.
As $R$ is Noetherian, $I$ is finitely generated, so it is nilpotent.
Moreover, since $I$ is a quotient of $R^k$ for some $k$, for each $j\in \bbN$, $I^j \cong R^{k^j}/M_j$ for some $R$-module $M_j$.
Therefore, we have right exact sequences
$$I^j\otimes I\to I^j\to I^j\otimes (R/I)\to 0$$
and
$$(R/I)\otimes M_j\to (R/I)^{k^j}\to I^j\otimes (R/I)\to 0.$$
If $R/I$ is finite, the same is true of $I^j\otimes (R/I)$ for all $j$, and then
$$I^{j+1} = \mathrm{im}\,(I^j\otimes I\to I^j)$$
is of finite index in $I^j$ for all $j$.  This implies that $(0) = I^n$ is of finite index in $R$, contrary to assumption.
Thus $R/I$ is infinite, which, since all maximal ideals have finite residue fields, implies that $R$ has infinitely many many maximal ideals.

The image of the
morphism $\Spec R\to \Spec \bbZ$ is constructible, so either it contains
the generic point or it is finite.  In the former case, let $P$ denote a prime ideal of $R$ lying over $(0)$.
The field of fractions of $R/P$ is a finitely generated extension of $\bbQ$ and therefore
embeds in $\bbC$.  It follows \cite[Lemma 1.4]{Larsen-IJM} that there exists infinitely many rational primes $p$
such that $\bbF_p$ is a quotient of $R$.  If $\m$ is the kernel of a homomorphism $R\to \bbF_p$,
then the natural injective homomorphism $R_0/\phi^{-1}(\m)\to R/\m\cong \bbF_p$ is an isomorphism.

We may therefore assume that $R$ lies over a finite set of primes $\{p_1,\ldots,p_k\}$.  As $R$ is an integral domain,
it is an $\bbF_p$-algebra for some $p=p_i$, so the same is true of $R_0$.  Let $X:= \Spec R$ and $X_0 := \Spec R_0$.
Thus $X\to X_0$ is a dominant finite morphism, and it follows that $X$ and $X_0$ have the same dimension,
which we denote $d$.
Let $\bbF_q$ denote a finite extension of $\bbF_p$ such that
all irreducible components of $X\times \Spec\bbF_q$ and $X_0\times \Spec\bbF_q$ are geometrically
irreducible.  By the Lang-Weil estimate, $|X(\bbF_{q^k})|$ is bounded below by $q^{dk}/2$ for large $k$,
while
$$\sum_{\bigl\{j<k\bigm|\frac kj\in\bbN\bigr\}} |X_0(\bbF_{q^j})| = O(q^{dk/2}).$$
It follows that for every sufficiently large $k$, there exists $x\in X(\bbF_{q^k})$ which maps to a point
of $X_0$ with residue field $\bbF_{q^k}$.  This proves the proposition.
\end{proof}

\begin{thm}
\label{Main}
Let $X$ be a split hyperlliptic curve over a field $K$, and $\phi\colon X\to A$ a non-constant $K$-morphism from $X$ to an abelian variety $A$.
Suppose that $K$ is not locally finite, $\ch K\neq 2$, and  $G_K$ is  finitely generated.  Then the rank of $A$ over $K$ is infinite.
\end{thm}

\begin{proof}

Let $X$ have an affine open set of the form
$$y^2 = (x-a_1)\cdots(x-a_{2g+2}).$$
We define finite sets $C$ and $L$ as in Proposition~\ref{MainProp}.

Let $R_0$ denote a finitely generated $\bbZ$-subalgebra of $K$ such that $Z$ and all of the $a_i$
are defined over $R_0$.  As $K$ is not locally finite, without loss of generality, we may assume that
$R_0$ is infinite.  As $\ch K\neq 2$, we may assume that $2$ is invertible in $R_0$.
Similarly, as the $a_i$ are pairwise distinct in $K$, we may assume $a_i-a_j$ is invertible in $R_0$ for
$1\le i < j \le 2g+2$.  Likewise, we may assume that the $t$-coefficient of every $\lambda\in L$ is invertible.
Let $F_0$ denote the fraction field of $R_0$.  If $F$ is any finite extension of $F_0$, the integral
closure $R$ of $R_0$ in $F$ is a finitely generated $R_0$-module (since $R_0$ is finitely generated over $\bbZ$), and therefore a
finitely generated $\bbZ$-algebra.

As $2$ is invertible in $R$, for every maximal ideal $\m$, $q := |R/\m|$ is odd.  As $R$ is finitely
generated, for any fixed $q$, the set of homomorphisms $R\to \bbF_q$
is finite.  Therefore, all but finitely many maximal ideals have residue field
cardinality larger than any given constant.

By Lemma~\ref{Chebotarev}, there is an infinite sequence $\m_1,\m_2,\m_3,\ldots$ of
maximal ideals of $R$ with $q_i := |R/\m_i|> c$ for all $i$, where $c:=2g^2+2g+3+2g\lceil\sqrt{g^2+2g+3}\rceil$, and such that $|R_0/R_0\cap \m_i|=q_i$.    The reduced equation
$$y^2 = (x-\bar a_1)\cdots(x-\bar a_{2g+2})$$
over $R/\m_i$ is non-singular
and of genus $g$, so it has at most $q+2g\sqrt q+1$ $\bbF_{q_i}$-points.
On the other hand, the $x$-coordinate gives a map
to the affine line which is two-to-one except at $2g+2$ points.  For $q_i>c$, 
we have
$2g\sqrt {q_i} < q_i-2g-3$, so there are fewer than $2q_i-2g-2$ $\bbF_{q_i}$-points on our reduced affine curve.
It follows that there exists $x_i\in \bbF_{q_i}$ such that
$(x_i-\bar a_1)\cdots(x_i-\bar a_{2g+2})$ has no square root in $\bbF_{q_i}$.

By the coprimality of
maximal ideals, the natural homomorphism
$$R_0\to \prod_{i=1}^n R_0/(\m_i\cap R_0)$$
is surjective for all $n$.
Writing $L = \{\lambda_1,\ldots,\lambda_N\}$ we can choose $c\in R_0\setminus C$ such that
$\lambda_i(c)$ reduces (mod $\m_i$) to $x_i$.  This implies that $\lambda_i(c)$ cannot be the $x$-coordinate
of an $F$-point of $E$ for any $i$.  On the other hand, by Proposition~\ref{MainProp}, for some $i$,
$\lambda_i(c)$ is the $x$-coordinate of some point of $X(K)$.  This point can be defined over a quadratic extension $F_i\subset K$
of $F_0$.

We can now proceed iteratively, finding a sequence of points $P_n\in X(F_i)$, such that $\phi(P_n)\not\in A(F_1\cdots F_{n-1})$.
To find $P_1$, we define $F=F_0$.
To find $P_n$ for $n\ge 2$, we define $F$ to be the compositum $F_1\cdots F_{n-1}$.
By Proposition~\ref{Silverman}, $\phi(P_n)\otimes 1$ is non-zero in $A(K)\otimes\bbQ$ for all $n$ sufficiently large.
Moreover, if $\phi(P_n)\otimes 1\neq 0$, it cannot lie in
$$\Span\{\phi(P_i)\otimes 1\mid 1\le i < n\}$$
because in
$A(F_1\cdots F_n)\otimes\bbQ\subset A(K)\otimes \bbQ$, this span is fixed pointwise by $\Gal(F_1\cdots F_n/F_1\cdots F_{n-1})$.

\end{proof}

\begin{cor}
If $K$ is not locally finite or of characteristic $2$ and $G_K$ is  finitely generated, then for an elliptic
curve $E/K$ such that $E[2](K)\cong (\bbZ/2\bbZ)^2$, the rank of $E$ over $K$ is infinite.
\end{cor}

\begin{proof}
As the $2$-torsion of $E$ is rational, $E$ is itself split hyperelliptic.  Setting $X = A = E$ and letting $\phi$ denote
the identity map, the corollary follows immediately from Theorem~\ref{Main}.
\end{proof}

\begin{cor}
If $F_0$ is of characteristic zero, then for every abelian surface $A/F_0$,
there exists a finite extension $F$ of $F_0$ such that for every extension $K$ of $F$ with
$G_K$ finitely generated, $A$ has infinite rank over $K$.  The same is true in characteristic $p>2$ if $F_0$
is not locally finite and $A$ is ordinary.
\end{cor}

\begin{proof}
Every abelian variety over an algebraically closed field is isogenous to a principally polarized abelian variety.
In characteristic zero, the ``Hecke  orbit'' of principally polarized abelian varieties
isogenous to a fixed $g$-dimensional principally polarized abelian variety $A$ is always Zariski-dense in the moduli space $A_g$,
and the same thing is true in characteristic $p>0$ if $A$ is ordinary \cite{Chai}.
However, by the Torelli theorem, there exists a non-empty open subset of the moduli space of principally polarized abelian varieties
which are Jacobians of curves of genus $2$.  Thus, every abelian surface admits a non-constant morphism from some hyperelliptic curve $X$,
defined over $\bar F_0$ and therefore over some finite extension $F$ of $F_0$.  Enlarging $F$ if necessary, we may assume that $X$ is split and apply
Theorem~\ref{Main}.
\end{proof}

\end{document}